\documentclass[12pt,a4paper]{amsart}

\usepackage[utf8]{inputenc}
\usepackage[margin=1in]{geometry}
\usepackage{hyperref}
\usepackage{amsmath}
\usepackage{amsfonts}
\usepackage{amsthm}
\usepackage{amssymb}
\usepackage{xcolor}
\usepackage{graphicx}
\usepackage{mathtools}
\usepackage{tikz-cd}
\usepackage{ytableau}

\hypersetup{
    colorlinks=true,
    linkcolor=blue,
    filecolor=magenta,      
    urlcolor=cyan,
}
 
\theoremstyle{plain}
\newtheorem*{Th*}{Theorem}
\newtheorem{Th}{Theorem}[section]
\newtheorem{Cor}[Th]{Corollary}
\newtheorem*{Cor*}{Corollary}
\newtheorem{Prop}[Th]{Proposition}
\newtheorem{Lemma}[Th]{Lemma}

\newtheorem{Ex}[Th]{Example}

\theoremstyle{definition}

\newtheorem{Not}[Th]{Notation}
\newtheorem{Rmk}[Th]{Remark}

\newcommand{\vp}{\varphi}
\newcommand{\sm}{\setminus}

\newcommand{\im}{\operatorname{im}}

\newcommand{\Gl}{\operatorname{GL}}

\newcommand{\ann}{\operatorname{ann}}

\newcommand{\Hom}{\operatorname{Hom}}

\newcommand{\projdim}{\,\operatorname{projdim}\,}

\newcommand{\Ext}{\operatorname{Ext}}

\newcommand{\Sym}{\operatorname{Sym}}
\newcommand{\Soc}{\operatorname{Soc}}

\def\ZZ{\mathbb Z}

\def\KK{\mathbb K}
\def\CC{\mathbb C}
\def\PP{\mathcal{P}}
\def\DD{\mathcal{D}}

\def\maxm{\frak{m}}
\newcommand{\Rees}{\mathcal{R}}

\title{Local Cohomology of Certain Determinantal Thickenings}

\begin{document}
\author{Hunter Simper}\thanks{The author was partially supported by NSF grant DMS-2100288 and by Simons Foundation Collaboration Grant for Mathematicians \#580839}
\address{Hunter Simper \\ Department of Mathematics \\ Purdue University \\
West Lafayette \\ IN 47907 \\ USA} 
\email{hsimper@purdue.edu}
\maketitle

\begin{abstract}
Let $R=\CC[\{x_{ij}\}]$ be the ring of polynomial functions in $mn$ variables where $m> n$. Set $X$ to be the $m\times n$ matrix in these variables and $I:=I_n(X)$ the ideal of maximal minors of $X$.  We consider the rings $R/I^t$; for $t\gg 0$ the depth of $R/I^t$ is equal to $n^2-1$, and we show that each local cohomology module $H^{n^2-1}_\maxm(R/I^t)$ is a cyclic $R$-module. We also compute the annihilator of $H^{n^2-1}_\maxm(R/I^t)$ thereby completely determining its $R$-module structure.

In the case that $X$ is a $n\times (n-1)$ matrix we describe a map between the Koszul complex of the $t$-powers of the maximal minors and a free resolution of $R/I^t$. We use this map to explicitly describe the modules $\Ext_R ^n(R/I^t,R)$ as submodules of the top local cohomology module $H_I^n(R)$. Moreover, we can realize the filtration $\bigcup_i\Ext_R ^n(R/I^t,R)= H_I^n(R)$ in terms of differential operators. Utilizing this description, along with an explicit isomorphism $H_I^n(R) \cong H_\maxm^{n(n-1)}(R)$, we determine the annihilator of $\Ext_R ^n(R/I^t,R)$ and hence by graded local duality give another computation of the annihilator of $H^{(n-1)^2-1}_\maxm(R/I^t)$. 
\end{abstract}

\tableofcontents

\section{Introduction}
Let $I$ be a homogeneous ideal in a polynomial ring $R$. Then $I$ defines a projective variety and one may consider its thickenings, i.e., the varieties defined by the ideals $I^t$. Understanding the ideals $I^t$ is an important component of understanding the singularities of the variety defined by $I$. For example, they comprise the graded components of Rees algebras and also appear in the study of the functors $H_I^i(-)$. 
It was shown in \cite{BBLZ19Stabilizationofthecohomologyofthickenings} that under certain conditions the graded components of the local cohomology modules $H_\maxm^i(R/I^t)$ stabilize for sufficiently large $t$. This recent work has brought renewed attention to thickenings and created an interest in their homological properties and invariants. 

In the case that $I=I_r(X)$ and $R=\CC[X]$ where $X$ is matrix of indeterminates, the modules $H^i_\maxm(R/I^t)$, $\Ext_R ^i(R/I^t,R)$ and $H_I^i(R)$ have been studied extensively and successfully using representation theoretic techniques. In \cite{Raicu-Weyman-Witt14Localcohomologywithsupportinidealsofmaximalminorsandsub-maximalPfaffians},
\cite{RaicuWeyman14Localcohomologywithsupportingenericdeterminantalideals} and \cite{Raicu18Regularityandcohomologyofdeterminantalthickenings} Raicu--Weyman--Witt, Raicu--Weyman and Raicu described the $\Gl$-equivariant structure of $\Ext_R ^i(R/I^t,R)$ and $H_I^i(R)$. These results have been used by Kenkel and Li in \cite{Kenkel20LengthsofLocalCohomologyofThickenings} and \cite{Li21Anoteonthemultiplicitiesofthedeterminantalthickeningsofmaximalminors} to study the asymptotic length of $H_\maxm^i(R/I^t)$ and find formulas for the higher epsilon multiplicity of $I$. In a similar flavor, the regularity of $I^t$ was described in \cite{RaicuWeyman14Localcohomologywithsupportingenericdeterminantalideals} and \cite{Raicu18Regularityandcohomologyofdeterminantalthickenings} along with a classification of which $\Gl$-invariant ideals satisfy the property that $H_I^i(R)=\bigcup_t \Ext^i(R/I^t,R)$. 

In this paper we focus on the case that $I\subseteq \CC[X]$ is the ideal of maximal minors of a $m\times n$ generic matrix $X$, with $m>n$ and examine $H^{n^2-1}_\maxm(R/I^t)$. For sufficiently large $t$, $H^{n^2-1}_\maxm(R/I^t)$ is the first non-vanishing local cohomology module of $R/I^t$ and it was shown by Li that $n^2-1$ is the only cohomological index to yield a nonzero finite length module \cite{Li21Anoteonthemultiplicitiesofthedeterminantalthickeningsofmaximalminors}. In Proposition \ref{prop-cyclic R module, general case}, we will show that $H^{n^2-1}_\maxm(R/I^t)$ is in fact a cyclic $R$-module. This module has also been examined in the case that $X$ is $2\times 3$ matrix in \cite{Kenkel20IsomorphismsBetweenLocalCohomologyModulesAsTruncationsofTaylorSeries} where Kenkel explicitly describes a generator of $[H^3_\maxm(R/I^t)]_0$ via the \v{C}ech complex on the variables of $R$.

The aforementioned results about $H_\maxm^{i}(R/I^t)$ speak about the structure of its graded components, i.e., its structure as a graded $\CC$-vector space. Additionally, the description of $\Ext_R ^i(R/I^t,R)$ given in \cite{Raicu-Weyman-Witt14Localcohomologywithsupportinidealsofmaximalminorsandsub-maximalPfaffians} is as a $\Gl$-representation and a priori does not speak on its structure as an $R$-module. In this paper we study the $R$-module structure of these modules and explicitly describe this structure for certain $\Ext$ and local cohomology modules.

We proceed by investigating the modules $\Ext_R ^i(R/I^t,R)$ via the natural map \\
$\Ext_R ^i(R/I^t,R)\to H_I^i(R)$. In the case that $I$ is the ideal of maximal minors of a generic matrix, the natural map is an injection \cite{Raicu-Weyman-Witt14Localcohomologywithsupportinidealsofmaximalminorsandsub-maximalPfaffians}, hence describing $\Ext_R ^i(R/I^t,R)$ is equivalent to describing its image in $H_I^i(R)$. 
To understand the map $\Ext_R ^i(R/I^t,R)\to H_I^i(R)$ we first view $H_I^i(R)$ as \v{C}ech cohomology of the maximal minors and compare this to the Koszul cohomology of the powers of maximal minors in the usual way. We then examine the natural map from $\Ext_R^i(R/I^t,R)$ to this Koszul cohomology. As the map from Koszul cohomology to \v{C}ech cohomology is well understood, it remains to understand the map $\Ext_R ^i(R/I^t,R) \to H^i([d_1^t \ldots d_k^t];R)$ where $d_1,\cdots,d_k$ are the maximal minors of $X$. Thus, to explicitly describe $\Ext_R ^i(R/I^t,R) \to H^i([d_1^t \ldots d_k^t];R)$ we need to describe a map of complexes $\vp_t$ such that
 
\begin{equation*}
\begin{tikzcd}
F_\bullet \ar{r} &  I^t \ar{r} & 0  \\
K_{\bullet} ([d_1^t \ldots d_k^t];R) \ar{r} \ar{u}{\vp_t} & (d_1^t, \cdots,d_k^t) \ar[hookrightarrow]{u} \ar{r}& 0
\end{tikzcd}
\end{equation*}
commutes, where $F_\bullet$ is a free resolution of $I^t$. The utility of using this approach to study $\Ext_R ^i(R/I^t,R)$ is in that the module structure of $H_I^i(R)$ may be quite familiar, cf. \cite[Main Theorem]{Raicu-Weyman16LocalcohomologywithsupportinidealsofsymmetricminorsandPfaffians}. For example, for $i=mn-n^2+1$, the cohomological dimension of $I$, $H_I^i(R)\cong H_\maxm^{mn}(R)$.

In the case that $X$ is size $n\times (n-1)$ we are able to explicitly construct a map $\vp_t$ as above; this is the content of Section \ref{section-lift}. Using this lift, in Section \ref{section-n times n-1 full section }, we give the following description of $\Ext_R ^{n}(R/I^t,R)$ as a submodule of $H^{n}_I(R)$.

\begin{Th*}[\ref{Th-Ext gens}] Let $X$ be a $n\times (n-1)$ matrix of indeterminates and $R=\CC[X]$. Set $I=(d_1,\ldots,d_n)\subseteq R$ where $d_1,\ldots,d_n$ are the maximal minors of $X$. For a tuple $\alpha=(\alpha_1,\ldots,\alpha_n)\in \ZZ^n_{\geq 0}$ write $d^\alpha=d_1^{\alpha_1}\cdots d_n^{\alpha_n}$. Then $\Ext_R^n(R/I^t,R)$ embeds into $H_I^n(R)$ as the submodule generated by the classes
$$\left\{\frac{1}{\prod_{i=1}^n d_i} \cdot  \frac{1}{d^\alpha} \right\}_{|\alpha|=t-n+1}.$$ 
\end{Th*}

This embedding can be realized as coming from differential operators and after identifying $H^n_I(R)$ with $H_\maxm^{n(n-1)}(R)$ we obtain the following key corollary.

\begin{Cor*}[\ref{prop-Ext in H_m}]
In the setting of the previous theorem, for $f\in R$, let $f^*$ denote the polynomial differential operator obtained from $f$ by replacing $x_i$ with $\partial_{i}$. Then for $t\geq n-1$ we have that $\Ext_R ^n(R/I^t,R)$ embeds in $H_\maxm^{n(n-1)}(R)$ as the $R$-submodule generated by the classes 
\[\left \{ (d^\alpha)^* \bullet \frac{1}{\underline{x}} \right\}_{|\alpha|=t-n+1}, \]
where $\bullet$ denotes the application of an operator.
\end{Cor*}


The Weyl algebra annihilator of the class $\frac{1}{\underline{x}}\in H_\maxm^{n(n-1)}(R)$ is well understood and in the remainder of Section \ref{section-n times n-1 full section } we use Corollary \ref{prop-Ext in H_m} to compute the $R$-annihilator of $\Ext_R^n(R/I^t,R)$. By graded duality, the annihilator of $\Ext_R^n(R/I^t,R)$ is the annihilator of $H_\maxm^{(n-1)^2-1}(R/I^t)$, hence we obtain a complete description of $H_\maxm^{(n-1)^2-1}(R/I^t)$ when $X$ is size $n\times (n-1)$ as a cyclic $R$-module generated in degree zero, see Proposition \ref{prop-cyclic R module, general case}.

In the general case of maximal minors of an $m\times n$ matrix with $m>n+1$ the map of complexes, $\vp_t$, and with it the structure of the $\Ext$ modules, remains mysterious. However, by analyzing the $\Gl$-structure of $H_\maxm^{n^2-1}(R/I^t)$ we are able to compute its annihilator as follows:

\begin{Th*}[\ref{Th-J_t description nx(n-1)}, \ref{Th-J_t description}]
Let $X$ be a $m\times n$ matrix of indeterminates with $m>n$ and set $I=I_n(x)\subseteq R=\CC[X]$. If $t< n$ then $H_\maxm^{n^2-1}(R/I^t)=0$. If $t\geq n$, then we have an isomorphism of graded $R$-modules:
\[H_\maxm^{n^2-1}(R/I^t)\cong R/ I_\lambda,\]
where $I_\lambda$ is the $\Gl$-invariant ideal associated to the partition $\lambda=(t-n+1)$, i.e., the ideal generated by $\Gl_m\times \Gl_n$ orbit of $x_{1,1}^{t-n+1}$, i.e., the ideal of $t-n+1$ generalized permanents of $X$ c.f. \ref{Ex-Gl ideal}.
\end{Th*}

\section*{Acknowledgements}

I would like to thank Claudiu Raicu for his helpful comments and suggestions, particularly in regards to Section \ref{section-general case} and tracking $\Gl$-structure across graded local duality. I am grateful to Bernd Ulrich and Claudia Polini for the conversations on Rees algebras. My sincerest thanks go to my advisor Uli Walther for the many discussions and feedback provided. 


\section{Background}

\begin{Not}
Let $R=\CC[x_1,\ldots,x_n]$ be a polynomial ring and $\DD=R[\partial_{1},\ldots,\partial_{n}]$ be the ring of differential operators on $R$. Fix $f\in R$ and $\psi\in \DD$. 
\begin{itemize}
    \item For a $\DD$-module $M$ and an element $h\in M$ we write $\psi \bullet h$ for element obtained by acting on $h$ by $\psi$. In particular $\psi \bullet f\in R$ is the application of $\psi$ to $f$.
    \item We write $\psi f \in \DD$ for the multiplication of $\psi$ and $f$ in $\DD$. 
    \item We write $f^*\in \DD$ for $f(\partial)$, the ``dual" operator to $f$ obtained by replacing $x_i$ by $\partial_{i}$.
\end{itemize}
Let $G$ be a group acting on a set $S$. Fix $g\in G$ and $s\in S$.
\begin{itemize}
    \item We write $g\cdot s$ for the element obtained by acting on $s$ with $g$.
    \item We write $G\cdot s$ for the orbit of $s$.
\end{itemize}
\end{Not}

\subsection{Dominant Weights, Partitions and Schur Functors}
We begin by establishing some notation and recalling some useful facts about Schur functors, for a complete treatment see \cite{Fulton-Harris91Representationtheory} and \cite{Weyman03Cohomologyofvectorbundlesandsyzygies}.  
A vector $\lambda=(\lambda_1,\ldots,\lambda_n)\in \ZZ^n$ is called a \textit{dominant weight} if $\lambda_1\geq \lambda_2 \geq \cdots \geq \lambda_n$. We write $\ZZ^n_{dom}$ for the set of all dominant weights in $\ZZ^n$ and write $|\lambda|=\sum_{i=1}^n \lambda_i$ for the \textit{size} of $\lambda$. Additionally, for $c\in\ZZ$ and $0 \leq d \leq n$, we write $(c^d)\in \ZZ^n_{dom}$ for the vector with $d$ nonzero components all equal to $c$.  A \textit{partition} into $n$ parts is a dominant weight, $\lambda\in\ZZ^n_{dom}$, with $\lambda_n \geq 0$, we write $\PP_n\subseteq \ZZ^n_{dom}$ to be the set of all such weights. An element $\lambda \in \PP_n$ may be realized as a \textit{Young diagram} with $\lambda_i$ boxes in row $i$, for example the diagram associated to $(4,3,1)\in \PP_3$ is:
\begin{equation}\label{yd-431}
   \ytableausetup{centertableaux}
\begin{ytableau}
 \phantom{d}& & & \phantom{d} \\
\phantom{d}  & &  \\
\phantom{d}
\end{ytableau}. 
\end{equation}

If $m\geq n$, we can naturally identify an element of $\PP_n$ with an element of $\PP_m$ by adjoining zeroes, e.g., $(2,2)\in \PP_2$ is identified with $(2,2,0,0)\in\PP_4$. Generally we omit the trailing zeroes and would write $(2,2)\in \PP_4$. For $\lambda \in \PP_n$ we can consider its \textit{transpose partition}, which is the partition associated to the transpose of the Young diagram of $\lambda$. The transpose partition of $(4,3,1)$ is $(3,2,2,1)$ because the transpose of \eqref{yd-431} is:
\begin{equation}
   \ytableausetup{centertableaux}
\begin{ytableau}
 \phantom{d}&  & \phantom{d} \\
\phantom{d}  &   \\
\phantom{d}  &   \\
\phantom{d}
\end{ytableau}. 
\end{equation}

Let $H$ be an $n$ dimensional $\CC$-vector space.  Then to each dominant weight $\lambda\in\ZZ^n_{dom}$ we associate an irreducible representation of $\Gl(H)$, denoted $S_\lambda H$, called a \textit{Schur functor}. Moreover every irreducible representation of $\Gl(H)$ can be realized in this manner. For some dominant weights, Schur functors are quite familiar: there are $\Gl(H)$-equivariant isomorphisms:

\[ S_{(1^d)} H\cong \bigwedge^d H  \]
and 
\[S_{(d)} H\cong \Sym^d H. \]

For computational purposes, frequently it is sufficient to consider $\lambda \in \PP_n$ as we have the following $\Gl(H)$-equivariant isomorphisms: 

\[S_{\lambda+(1^n)}H \cong S_\lambda H \bigotimes \bigwedge^n H \] and 
\[S_{(\lambda_1,\ldots, \lambda_n)}H \cong \Hom(S_{(-\lambda_n,\ldots,-\lambda_1)}H,\CC).\]
Lastly, the dimension of Schur functors is relatively easy to compute and is described by the following formula.

\begin{Prop}\cite{Fulton-Harris91Representationtheory}
Let $\lambda\in \ZZ^n_{dom}$, then
\[\dim_\CC S_\lambda \CC^n = \prod_{1\leq i < j \leq n} \frac{\lambda_i-\lambda_j+j-i }{j-i}. \]
 
\end{Prop}

\subsection{$\Gl$-invariant ideals}\label{section-Gl invariant ideals}
Let $F=\CC^m$ and $G=\CC^n$ where $m\geq n$. Then 
\[R:=\Sym(F\otimes G)=\CC[\{x_{ij}\}]=\CC[X]\]
is a polynomial ring admitting a 
\[\Gl:= \Gl(F)\times \Gl(G)\]
action as follows: for $(g_1,g_2)\in \Gl$, $g\cdot (x_{ij})=(z_{ij})$ where $[z_{ij}]=g_1 X g_2^{-1}$. \textit{Cauchy's formula} describes the decomposition of $R$ into irreducible representations \cite{Weyman03Cohomologyofvectorbundlesandsyzygies}:
\begin{equation}\label{eq-Cauchy's formula}
    R=\bigoplus_{\lambda\in \PP_n} S_\lambda F \otimes S_\lambda G
\end{equation}
where  $S_\lambda F \otimes S_\lambda G$ lives in degree $|\lambda|$.

For a number $1\leq l \leq n$ set $\det_l:=\det(x_{ij})_{1\leq i,j\leq l}$ , i.e., the $l\times l$ minor in the top left corner of $X$. Then for a partition, $\lambda\in \PP_{n}$, with $n$ parts let $\lambda'$ be the transpose partition and define
\[{\det}_\lambda :=\prod_{i=1}^{\lambda_1} {\det}_{\lambda'_i}.\]
The $\CC$-linear span of the $\Gl$ orbit of $\det_\lambda$ is equal to $S_\lambda F \otimes S_\lambda G$. We set
\[I_\lambda := (\Gl \cdot \; {\det}_\lambda)\subseteq R,\] 
the ideal generated by the $\Gl$ orbit of $\det_\lambda$. This ideal is $\Gl$-invariant. We endow $\PP_n$ with a partial ordering: for $\mu,\lambda\in \PP_n$, we say that $\mu \geq \lambda$ if $\mu_i \geq \lambda_i$ for all $i$. It was shown in \cite{deConcini-Eisenbud-Procesi80Youngdiagramsanddeterminantalvarieties} that:
\begin{equation}\label{eq-ideal generated by single partition}
 I_\lambda=\bigoplus_{\mu\geq \lambda} S_\mu F \otimes S_\mu G.   
\end{equation}

 By taking a collection of partitions $\chi\subseteq \PP_{n}$ we can form the $\Gl$-invariant ideal $I_\chi=\sum_{\lambda\in \chi} I_\lambda$. It was proven in \cite{deConcini-Eisenbud-Procesi80Youngdiagramsanddeterminantalvarieties} that all $\Gl$-invariant ideals may be realized in this manner for some finite subset $\chi\subseteq \PP_{n}$ and so more generally $\Gl$-invariant ideals decompose as:
 \[I_\chi=\bigoplus_{\substack{\mu\geq \lambda \\ \lambda \in \chi}} S_\mu F \otimes S_\mu G.\]

 \begin{Ex}\label{Ex-Gl ideal} Let $r,t$ be positive integers.
 \begin{enumerate}
     \item[(1)] $I_{(1^r)}=I_r(X)$ the ideal of $r\times r$-minors of $X$.
     \item[(2)] $I_{\chi_t} =I_r(X)^t$ where $\chi_t=\{\lambda\in \PP_n |\; |\lambda|=rt, \lambda_1\leq t \}.$
     \item[(2$^\prime$)] $I_{(t^n)}=I_n(X)^t$ the $t$-th power of the ideal of maximal minors of $X$.
     \item[(3)] $I_{(t)}$, the ideal of $t\times t$ \textit{generalized permanents} of $X$, i.e., the ideal generated by the permanent of all $t\times t$ matrices of the form $[x_{\alpha_i, \beta_j}]$ where $\alpha_i \leq \alpha_{i+1}$ and $\beta_j \leq \beta_{j+1}$.
 \end{enumerate}
 \end{Ex}

 \begin{Rmk}\label{rmk-Gl-ideals determined by rep structure}
 Notice that Cauchy's Formula, \eqref{eq-Cauchy's formula}, says that every irreducible representation of $\Gl$ contained in $R$ appears at most once. Combining this with the classification of $\Gl$-invariant ideals of \cite{deConcini-Eisenbud-Procesi80Youngdiagramsanddeterminantalvarieties} stated above, we have that a $\Gl$-invariant ideal is uniquely determined by its structure as a $\Gl$-representation.  
 \end{Rmk}

\subsection{A $\CC$-Linear $\Gl$-Equivariant Pairing}\label{section-equivariant pairing}
Let $R$ be as above. Let $R^*=\CC[\partial_{ij}]$ and 
\[(-)^*: R \to R^*\] be the map induced by $x_{ij}\mapsto \partial_{ij}$. We can view $R$ as the coordinate ring for the space of $m\times n$ complex matrices and the $\Gl$ action on $R$ described above as being induced by the $\Gl$ action on this space of matrices. We now view $R^*$ as the coordinate ring for the space of $n \times m$ matrices and hence $\Gl$ acts on it as follows: for $g=(g_1,g_2)\in \Gl $, $g\cdot x_{ij}^* =z_{ij}$ where $[z_{ij}]=(g_1^{-1})^T X^* g_2^T $.\\

The action of $R^*$ on $R$ via differentiation induces a perfect pairing
\[\langle \;,\; \rangle :[R^*]_k\times [R]_k \to \CC.\] 
We will see below this pairing is $\Gl$-equivariant, where $\Gl$ acts on $R$ and $R^*$ as above and fixes $\CC$. A more general statement about differential operators acting on representations is known, see \cite[Section 2.2]{Lorincz-Raicu-Walther-Weyman17Bernstein-Satopolynomialsformaximalminorsandsub-maximalPfaffians}, however we include the following proof for completeness.

\begin{Lemma} The pairing, $\langle \;,\; \rangle :[R^*]_k\times [R]_k \to \CC$ is $\Gl$-equivariant.
\end{Lemma}
\begin{proof}
By linearity it is enough to check this for $f^*,g$ where $f=x^\alpha$ and $g=x^\beta$ are monomials. Let $\theta=(\theta_1,\theta_2)\in \Gl$, it is sufficient to prove the statement for $\theta =(\theta_1, ID)$ and $\theta=(ID,\theta_2)$, where $ID$ denotes the identity. The arguments in each case are analogous so we assume that $ \theta =(\theta_1,ID)$.
Thus, $\theta\cdot x_{ij}=\sum_{k=1}^n x_{ik}\theta_{kj} $ and $\theta\cdot x_{ij}^*=\sum_{k=1}^n x_{ik}^*(\theta^{-1})^T_{kj}= \sum_{k=1}^n x_{ik}^*\theta^{-1}_{jk}$.\\
We induct on $k$. For $k=1$ we need to consider $f=x_{ab}$, $g=x_{cd}$. Then $\langle f^*,g\rangle$ is $1$ if $(a,b)=(c,d)$ and $0$ otherwise. Now consider, 
\begin{align*}
  \langle \theta \cdot f^*, \theta \cdot g\rangle &= (\sum_{k=1}^n x_{ak}^*\theta^{-1}_{bk} ) \bullet (\sum_{k=1}^n  x_{ck}\theta_{kd}) \\
  &=\sum_{k=1}^n \sum_{l=1}^n \theta_{bk}^{-1} \theta_{ld}(x_{ak}^* \bullet x_{cl}) \\
  &=\begin{cases} 
      0 & a\neq c \\
      \sum_{k=1}^n  \theta_{bk}^{-1} \theta_{kd} &  \text{else} 
   \end{cases}\\
   &=\begin{cases} 
      0 & a\neq c \\
      ID_{b,d} &  \text{else} 
   \end{cases}\\
   &=\begin{cases} 
      1 & (a,b) = (c,d) \\
      0 &  \text{else} 
   \end{cases}.
\end{align*}
For $k>1$ we may write $f=x_{ab}f'$ and $g=x_{cd}g'$, then :
\begin{align}
    (\theta\cdot f^*) \bullet (\theta \cdot g)&=(\theta \cdot f'^*) \bullet (\theta \cdot x_{ab}^* \bullet (\theta \cdot g')(\theta \cdot x_{cd})) \label{is-1}\\
    &=(\theta \cdot f'^*) \bullet ((\theta \cdot x_{cd})(\theta \cdot  x_{ab}^* \bullet (\theta \cdot g'))+ (\theta \cdot g')(\theta \cdot  x_{ab}^* \bullet (\theta \cdot x_{cd}))) \label{is-2}\\
    &=(\theta \cdot f'^*) \bullet ((\theta \cdot x_{cd})(\theta \cdot  (x_{ab}^* \bullet g'))+ (\theta \cdot g')(\theta \cdot  (x_{ab}^* \bullet  x_{cd}))  \label{is-3}\\
     &=(\theta \cdot f'^*) \bullet (\theta \cdot (x_{cd}(x_{ab}^* \bullet g'))+ \theta \cdot( g'(\theta \cdot  (x_{ab}^* \bullet  x_{cd})))) \label{is-4}\\
     &=(\theta \cdot f'^*) \bullet (\theta \cdot (x_{cd}(x_{ab}^* \bullet g')+  g'(\theta \cdot  (x_{ab}^* \bullet  x_{cd}))) \label{is-5}\\
     &=\theta \cdot (f'^* \bullet  ( (x_{cd}(x_{ab}^* \bullet g')+  g'(\theta \cdot  (x_{ab}^* \bullet  x_{cd}))))) \label{is-6}\\
     &=\theta \cdot (f'^* \bullet  ( x_{ab}^* \bullet (x_{cd}g'))) \label{is-7}\\
     &=\theta \cdot (f^* \bullet g) , \label{is-8}
\end{align}
where \eqref{is-1} to \eqref{is-2} is by the product rule since $\theta\cdot x_{ab}^*$ is a linear operator. From \eqref{is-2} to \eqref{is-3} and \eqref{is-5} to \eqref{is-6} are by induction hypothesis. Finally \eqref{is-6} to \eqref{is-7} is again by the product rule.

\end{proof}

Now given an $\Gl$-invariant subspace of $[R]_k$, e.g., a graded component of a $\Gl$-invariant ideal, we can use this pairing to analyze which polynomial operators annihilate that subspace.

\begin{Lemma}\label{lemma-orbit kils equivariant subspace}
Let $N\subseteq[R]_k$ be a $\Gl$-invariant subspace and suppose $f\in [R]_j$ such that $f^* \bullet N=0$. Then the $\Gl$ orbit of $f$ also annihilates $N$. 
\end{Lemma}

\begin{proof}
If $j>k$ the statement is trivial so assume $j\leq k$.
Since $f^*\bullet N=0$ for all $g^*\in [R^*]_{k-j}$ and $h\in N$ we have that 
$\langle (fg)^*,h\rangle=0$. So for all $\vp \in \Gl$ we have that 
\[0=\langle \vp \cdot (fg)^*, \vp \cdot h\rangle=\langle (\vp \cdot f)^* (\vp \cdot g)^*,\cdot h\rangle.\]
Since the choice of $g$ was arbitrary and $N$ is $\Gl$-equivariant we have that for all $g^*\in [R^*]_{k-j}$ and all $h\in N$
\[\langle (\vp\cdot f)^* g^*,h\rangle=0.\]
Suppose for contradiction that $(\vp\cdot f)^* \bullet N\neq 0$, then there exists $h\in N$ such that $(\vp \cdot f)^* \bullet h \neq 0$. Thus, $(\vp \cdot f)^* \bullet h$ contains a monomial $g$ with nonzero coefficient, but then $\langle (\vp\cdot f)^* g^*,h\rangle \neq 0$, a contradiction.
\end{proof}

\subsection{$\Gl$-equivariant description of certain $Ext$ modules}\label{section-Gl equivariant description of Ext}
Let $R,F,G$ and $\Gl$ be as defined above in Section \ref{section-Gl invariant ideals}. Let $I=I_{(1^n)}$ be the ideal generated by the maximal minors of $X$. In  \cite{Raicu-Weyman-Witt14Localcohomologywithsupportinidealsofmaximalminorsandsub-maximalPfaffians} the authors gave a $\Gl$-equivariant description of $H_I^i(R)$ as a direct sum of irreducible $\Gl$-representations. They also prove a number of results about the modules $\Ext_R ^i(R/I^t,R)$, we recall two of these results below.
\begin{Th}\label{Th-RWW Ext graded pieces}\cite[Theorem 4.3]{Raicu-Weyman-Witt14Localcohomologywithsupportinidealsofmaximalminorsandsub-maximalPfaffians}
Let $m>n$ and $t\geq n$. 
If $\lambda=(\lambda_1,\ldots, \lambda_{n})\in \ZZ^{n}$ we write 
\[ \lambda(s)=(\lambda_1, \ldots, \lambda_{n-s},\underbrace{ -s,\ldots,-s}_{m-n}, \lambda_{n-s+1}+(m-n),\ldots,\lambda_{n} +(m-n))\in \ZZ^m\]
Writing $W(r;s)$ for the set of dominant weights $\lambda \in \ZZ_{dom}^{n}$ with $|\lambda|=r$ such that $\lambda(s)$ is also dominant. We have
\[[\Ext_R ^{s(m-n)}(I^t,R)]_{r} =\bigoplus_{\substack{\lambda\in W(r;s)\\ \lambda_{n} \geq -t-(m-n)}} S_{\lambda(s)} F \otimes S_\lambda G.\]

\end{Th}

An analogous description of $\Ext_R ^i(J,R)$ was computed in \cite{Raicu18Regularityandcohomologyofdeterminantalthickenings} for any $\Gl$-invariant ideal $J$ not just powers of maximal minors.
 
 We will use Theorem \ref{Th-RWW Ext graded pieces} in Section \ref{section-general case} in conjunction with graded duality to compute the $\Gl$-structure of $H_\maxm^{mn-n^2+1}(R/I^t)$. To make use of this description it will be useful to understand how the modules $Ext_R^i (R/I^t,R)$ sit inside $H_I^i(R)$.

\begin{Th}\label{Th-RWW injectivity of ext}\cite[Section 4]{Raicu-Weyman-Witt14Localcohomologywithsupportinidealsofmaximalminorsandsub-maximalPfaffians} Let $i\in \ZZ_{\geq 0}$, for all $t\geq 1$, the induced maps $\Ext_R^i (R/I^t,R)\to \Ext_R^i (R/I^{t+1},R)$ are injective.
\end{Th}
This immediately gives us the following:
\begin{Cor}\label{Cor-union} 
$$H_{I}^i (R) = \bigcup_{t\geq 0} \Ext_R^i (R/I^t,R).$$
\end{Cor}
 More generally, the pairs of $\Gl$-invariant ideals $I_{\chi_1}$ and $I_{\chi_2}$ for which 
 \[\Ext_R^i (R/I_{\chi_1},R)\to \Ext_R^i (R/I_{\chi_2},R)\] is injective is classified in \cite{RaicuWeyman14Localcohomologywithsupportingenericdeterminantalideals} and \cite{Raicu18Regularityandcohomologyofdeterminantalthickenings}. In particular, Theorem \ref{Th-RWW injectivity of ext} and Corollary \ref{Cor-union} fail for ideals of non-maximal minors.

\subsection{Other facts on local cohomology of determinantal things}
Let $R,I$ be as in Section \ref{section-Gl equivariant description of Ext} and let $\DD =R[\{\partial_{ij}\}]$ be the Weyl algebra. 
The action of differentiation makes $R$ a left $\DD$-module and formal application of the quotient rule then gives $R_a=R[\frac{1}{a}]$ a $\DD$-modules structure for all $a\in R$. Thus for any ideal $J=(a_1,\ldots,a_k)$, the \v{C}ech complex, $\check{C}^\bullet(a_1,\ldots, a_k; R)$, is a complex of $\DD$-modules, hence $H_J^i(R) \cong H^i(\check{C}^\bullet(a_1,\ldots, a_k; R))$ carries the structure of a $\DD$-modules.

\begin{Th}\label{Th-LSW}\cite[Theorem 1.2]{Witt12Localcohomologywithsupportinidealsofmaximalminors}\cite[Theorem 1.2, Theorem 3.1]{Lyubeznik-Singh-Walther16Localcohomologymodulessupportedatdeterminantalideals} There exists a degree preserving isomorphism of $\DD$-modules: 
$$H_{I_{r}(X)}^{mn-r^2+1}(R) \cong H_\maxm^{mn}(R),$$
in particular,
$$H_{I}^{mn-n^2+1}(R) \cong H_\maxm^{mn}(R).$$
\end{Th}
As these are cyclic $\DD$-modules, in order to describe an isomorphism as above, we just need to choose a socle generator of $H_{I}^{mn-n^2+1}(R)$ and $H_\maxm^{mn}(R).$ Such a map will be constructed in section \ref{section-spec to n times (n-1)} in the case where $X$ is size $n\times (n-1)$.

In light of Corollary \ref{Cor-union}, Theorem \ref{Th-LSW} and local duality gives us an avenue to examine $H_\maxm^{n^2-1}(R/I^t)$.  Some implications of this result to the asymptotic structure of the graded components of $H_\maxm^{n^2-1}(R/I^t)$ has been remarked on in \cite{BBLZ19Stabilizationofthecohomologyofthickenings} and \cite{Kenkel20IsomorphismsBetweenLocalCohomologyModulesAsTruncationsofTaylorSeries}. We obtain the following result on the structure of $H_\maxm^{n^2-1}(R/I^t)$ as an $R$-module.
\begin{Prop}\label{prop-cyclic R module, general case}
$H_\maxm^{n^2-1}(R/I^t)$ is a cyclic $R$-module generated in degree $0$. In other words, there exists $J\subseteq R$ such that $H_\maxm^{n^2-1}(R/I^t)\cong R/J$. 
\end{Prop}
\begin{proof}
First to show $H_m^{(n-1)^2-1}(R/I^t)$ is cyclic, by graded duality it is sufficient to show that $\Ext_R^{mn-n^2+1}(R/I^t,R)$ is finite length and has socle dimension at most $1$. By Theorem \ref{Th-RWW injectivity of ext} and Theorem \ref{Th-LSW} we have that $\Ext_R^{mn-n^2+1}(R/I^t,R)\hookrightarrow  H_{\maxm}^{mn}(R)$. Thus $\Ext_R^{mn-n^2+1}(R/I^t,R)$ is a finitely generated submodule of an Artinian module hence has finite length. Moreover, since $H_{\maxm}^{mn}(R)$ has socle dimension 1 we have that $\Ext_R^{mn-n^2+1}(R/I^t,R)$  has socle dimension at most $1$. That $H_m^{(n-1)^2-1}(R/I^t)$ is generated in dimension $0$ follows by graded duality since $\Soc (\Ext_R^{mn-n^2+1}(R/I^t,R))=\Soc(H_{\maxm}^{mn}(R))$ is generated in degree $-mn$.
\end{proof}



\section{A Map Between Complexes}\label{section-lift}
Let $A$ be a set, we write $\#A$ for the cardinality of $A$. Let $X$ be a $n \times (n-1)$ matrix of indeterminates. For $A\subseteq \{1,\ldots,n\}$ and $H\subseteq \{1,\ldots,n-1\}$ with $\#A=\#H$, we write $X_{A,H}$ for the determinant of the submatrix of $X$ coming with rows in $A$ and columns in $H$. We will make use of the Hilbert-Burch theorem so it is convenient for this section to use signed minors: set $\Delta_i = (-1)^i X_{\{i\}^c,\emptyset^c}$, that is to say $(-1)^i$ times the maximal minor of $X$ obtained by deleting the $i$th row.

Let $\KK$ be a field and $R=\KK[X]$, set $I=I_n(X)=(\Delta_1,\ldots,\Delta_n)\subseteq R$ the ideal of maximal minors of $X$. In this case the \textit{Rees algebra} of $I$, 
$$\Rees (I):=\bigoplus_{i\geq 0} I^i t^i \subseteq R[t],$$
is linear type, i.e., 
$$\Rees (I) \cong S/(F_1,\ldots,F_{n-1})$$ where $S=R[T_1 \cdots T_n]$ and $F_j=\sum_{i=1}^n x_{ij} T_i$. Moreover $\Rees (I)$ is a complete intersection so the Koszul complex of $[F_1, \cdots, F_{n-1}]: S^{n-1}\to S$ is a resolution. In this section we will compare the linear strands of this Koszul complex to the Koszul complexes of $[\Delta_1^t  \cdots \Delta_{n}^t]:R^n \to R$.

More precisely, we will describe maps of complexes, $\vp_r$, for each $r$ making the following diagram commute.
\begin{equation}\label{Diagram-lift}
\begin{tikzcd}
\left[K_\bullet ([F_1 \ldots F_{n-1}];S)\right]_t \ar{r} &  \left[\Rees(I)\right]_t \ar{r} & 0  \\
K_{\bullet-1} ([\Delta_1^t \ldots\Delta_{n}^t];R) \ar{r} \ar{u}{\vp_t} & (\Delta_1^t, \cdots,\Delta_{n}^t) \ar[hookrightarrow]{u} \ar{r}& 0.
\end{tikzcd}
\end{equation}
Where $\left[K_\bullet ([F_1 \ldots F_{n-1}];S)\right]_t$ denotes the $t$-th linear strand of $K_\bullet ([F_1 \ldots F_{n-1}];S)$.

First we need to establish some notation and prove a small lemma that will be helpful later.

\begin{Not}\;
\begin{enumerate}
    \item Let $A=\{a_1,\ldots,a_r\}\subseteq\ZZ_{\geq 1}$ with $a_1 < a_2<\ldots <a_r$. Then set
    \[e_A:=e_{a_1} \wedge e_{a_2} \wedge\ldots \wedge e_{a_r}\] and 
    \[\Delta_A:= \prod_{a\in A} \Delta_a.\]
    \item Let $A,B$ be ordered sets of integers then 
    \[\rho(A,B):=\begin{cases} 0 & \text{if } A\cap B \neq \emptyset\\
    (-1)^{\nu(A,B)} & \text{else}
    \end{cases} \]
    where $\nu(A,B)=\#\{(a,b)\subseteq A \times B | \; a>b\}$.
    \item Let $A\subseteq \{1,..,n\}$ and $H\subseteq \{1,..,n-1\}$ with $\#A=r$ and $\#H=r-2$. Then set 
    \[Y_{A,H,i} := \det \begin{bmatrix} 
    r_{i,H^c} \\
    Z
    \end{bmatrix},\]
    where $r_{i,H^c}$ is the entries of the i-th row of $X$ with columns in $H^c$ and $Z$ is the submatrix of $X$ with rows in $A^c$ and columns in $H^c$.
    \item For $A$ an ordered sets of integers set $(-1)^A:=(-1)^{\sum_{a\in A}a}$.
\end{enumerate}
\end{Not}

\begin{Lemma}\label{Lemma-ids}\;

\begin{enumerate}
    \item $$Y_{A,H,i}=\sum_{\alpha\in H^c} \rho(\{\alpha\},H^c\sm \{\alpha\}) x_{i\alpha} X_{A^c (H\cup {\alpha})^c}.$$
    
    \item Let $A\subseteq \{1,\ldots, m\}$ and $\alpha\in A$. Then, $$\rho(\{\alpha\},A\sm {\alpha}) \rho(\{\alpha\},A^c)=(-1)^{\alpha-1} .$$
\end{enumerate}
\end{Lemma}
\begin{proof}\;

\begin{enumerate}
    \item This is an expansion of the determinant along the first row.
    \item $\rho(\{\alpha\},A\sm {\alpha}) \rho(\{\alpha\},A^c)=(-1)^{\nu(\{\alpha\},A\sm {\alpha})+ \nu(\{\alpha\},A^c)}$. Now, 
    \[\nu(\{\alpha\},A\sm {\alpha})+\nu(\{\alpha\},A^c)=\#\{b\in A\sm \{\alpha\}| \alpha>b \} +\#\{b\in A^c| \alpha>b \}.\]
    Since $A\sm \{\alpha\}$ and $A^c$ are disjoint we have that
    \[\nu(\{\alpha\},A\sm {\alpha})+\nu(\{\alpha\},A^c)=\#\{b\in \{1\ldots m\}\sm \{\alpha\}| \alpha>b\}=\alpha-1.\]
\end{enumerate}
\end{proof}

Our strategy will be to consider each commutative square of Diagram \eqref{Diagram-lift} and induct on $t$. Theorem \ref{Th-base case} and Corollary \ref{Cor-extended base case} will constitute the base case of this induction with Theorem \ref{Th-base case} addressing the first $t$ for which a every module in a square is non-zero.

\begin{Th}\label{Th-base case}
Consider the following diagram for $n\geq r\geq 2$.
\[\begin{tikzcd}
\bigwedge^{r-1} (R1)^{n-1} \ar{r}{\delta} & \bigwedge^{r-2} (\bigoplus RT_i)^{n-1} \\
\bigwedge^r R^n \ar{r}{d} \ar{u}{\vp_{r-1}^{r-1}}& \bigwedge^{r-1} R^n \ar{u}{\vp_{r-1}^{r-2}},
\end{tikzcd}\]
where $\delta$ is the map on the $(r-1)$-st linear strand of $K_\bullet([F_1\ldots F_{n-1}];S)$ and the bottom map, $d$, is the differential of $K([\Delta_1^{r-1}\ldots\Delta_n^{r-1}];R)$. 

Let $(f_j)_{j=1}^{n-1}$ denote the standard $S$-basis of $S^{n-1}$ and $(e_i)_{i=1}^{n}$ denote the standard $R$-basis for $R^n$. Define the vertical maps as follows:  let $A,B\subseteq \{1,\ldots,n\}$ with $\#A=r$ and $\#B=r-1$, set

\begin{align*}
    \vp_{r-1}^{r-1}(e_A):=(-1)^{r-1} \Delta_A^{r-2} \sum_{\substack{K\subseteq\{1\ldots n-1\} \\ \#K=r-1}} (-1)^{A+K} X_{A^c K^c} f_K 
\end{align*}

\begin{align*}
 \vp_{r-1}^{r-2}(e_B) &:= (-1)^{r} \Delta_B^{r-3} \sum_{\substack{K\subseteq\{1\ldots n-1\} \\ \#K=r-2}} (-1)^{B+K} X_{B^c K^c} f_K \sum_{b\in B} \Delta_B\frac{T_b}{\Delta_b} \\
 &= (-1)^{r} \Delta_B^{r-2} \sum_{b\in B} \frac{T_b}{\Delta_b} \sum_{\substack{K\subseteq\{1\ldots n-1\} \\ \#K=r-2}} (-1)^{B+K} X_{B^c K^c} f_K 
\end{align*}
Then the diagram above commutes. (Note that for $r=2$, $\vp_{r-1}^{r-2}$ is the map that takes $e_i \to T_i$).
\end{Th}
Before we begin the proof we first give an example.

\begin{Ex}\label{ex-lift1} Suppose $n=3$ and $t=2$ then Diagram \eqref{Diagram-lift} becomes:
\[\begin{tikzcd}
  0\ar{r}& \left[\bigwedge^2 S^2 \right]_0 \ar{r}& \left[\bigwedge^1 S^2 \right]_1 \ar{r}& \left[\bigwedge^0 S^{2}\right]_2 \ar{r} & I^2 \ar{r} & 0\\
  0 \ar{r}& \bigwedge^3 R^3 \ar{r} \ar{u}{\vp^2_2}& \bigwedge^3 R^2 \ar{r} \ar{u}{\vp^1_2}& \bigwedge^1 R^3 \ar{r} & (\Delta_1^2,\Delta_2^2,\Delta_3^2) \ar[hookrightarrow]{u} \ar{r}&0. 
\end{tikzcd}
\]
Then, Theorem \ref{Th-base case} says the left square commutes for,
\begin{align*}
    \vp_2^2(e_1 \wedge e_2 \wedge e_3)&= (-1)^2 \Delta_1  \Delta_2 \Delta_3( (-1)^{(6+3)}f_1 \wedge f_2)\\
    &=- \Delta_1  \Delta_2 \Delta_3(f_1 \wedge f_2)
\end{align*}  
and
\begin{align*}
    \vp_2^1(e_a \wedge e_b)&=(-1)^{3-2} \Delta_a \Delta_b \left( \frac{T_a}{\Delta_a}+\frac{T_b}{\Delta_b}\right)( (-1)^{a+b+1}x_{\{a,b\}^c,2} f_1 + (-1)^{a+b+2}x_{\{a,b\}^c,1} f_2)\\
    &=-( \Delta_b T_a + \Delta_a T_b)( (-1)^{a+b+1}x_{\{a,b\}^c,2} f_1 + (-1)^{a+b+2}x_{\{a,b\}^c,1} f_2)\\
    &=\begin{cases}  -( \Delta_2 T_1 + \Delta_1 T_2)( x_{3,2} f_1 -x_{3,1} f_2 )& (a,b)=(1,2),\\
    -( \Delta_3 T_1 + \Delta_1 T_3)( -x_{2,2} f_1 + x_{2,1} f_2)& (a,b)=(1,3),\\
    -( \Delta_3 T_2 + \Delta_2 T_3)( x_{1,2} f_1  -x_{1,1} f_2)& (a,b)=(2,3).\\
    \end{cases}
\end{align*}
One may notice that the only way to possibly complete this diagram with a map $\vp_2^0: \bigwedge^1 R^3 \to [\bigwedge^0 S^2]_2$ and have any hope that it commutes is to set $\vp_2^0(e_i)=T_i^2$. Later, in Theorem \ref{Th-full lift} we will see that this is the correct choice to make the diagram commute, along with how to construct the maps for other $t$.
\end{Ex}

\begin{proof}[Proof of Theorem \ref{Th-base case}]
To show this diagram commutes we simply compute the two compositions of maps. Fix $A \subseteq \{1\ldots n\}$. Then,
\begin{align*}
    \delta(\vp_{r-1}^{r-1}(e_A))&=\delta((-1)^{r-1} \Delta_A^{r-2} \sum_{\substack{K\subseteq\{1\ldots n-1\} \\ \#K=r-1}} (-1)^{A+K} X_{A^c K^c} f_K)\\
    &=(-1)^{r-1} \Delta_A^{r-2} \sum_{\substack{K\subseteq\{1\ldots n-1\} \\ \#K=r-1}} (-1)^{A+K} X_{A^c K^c} \delta(f_K)\\
    &=(-1)^{r-1} \Delta_A^{r-2} \sum_{\substack{K\subseteq\{1\ldots n-1\} \\ \#K=r-1}} (-1)^{A+K} X_{A^c K^c} (\sum_{k\in K} \rho(\{k\}, K \sm \{k\}) F_k f_{K\sm \{k\}})\\
    &=(-1)^{r-1} \Delta_A^{r-2} \sum_{\substack{K\subseteq\{1\ldots n-1\} \\ \#K=r-1}} (-1)^{A+K} X_{A^c K^c} (\sum_{k\in K} \rho(\{k\}, K \sm \{k\}) (\sum_{i=1}^n x_{ik} T_i) f_{K\sm \{k\}})\\
    &=(-1)^{r-1} \Delta_A^{r-2} \sum_{i=1}^n T_i \sum_{\substack{K\subseteq\{1\ldots n-1\}\\ \#K=r-1}} \sum_{k\in K}  (-1)^{A+K} X_{A^c K^c} \rho(\{k\}, K\sm \{k\}) x_{ik} f_{K\sm \{k\}}\\
    &=(-1)^{r-1} \Delta_A^{r-2} \sum_{i=1}^n T_i \sum_{\substack{H\subseteq\{1\ldots n-1\}\\ \#H=r-2}} (-1)^{A+H} f_H \sum_{\alpha \in H^c} (-1)^{\alpha} \rho(\{\alpha\},H)x_{i\alpha}X_{A^c (H \cup \{\alpha\})^c}  \\
\end{align*}
Now by  Lemma \ref{Lemma-ids} (2) we know that $\rho(\{\alpha\},H^c\sm {\alpha}) \rho(\{\alpha\},H)=(-1)^{\alpha-1}$. Hence, $(-1)^{\alpha} \rho(\{\alpha\},H)=(-1)\rho(\{\alpha\},H^c\sm {\alpha})$. So the above is
\begin{align*}
   \phantom{\delta(\vp_{r-1}^{r-1}(e_A))} 
   &=(-1)^{r-1} \Delta_A^{r-2} \sum_{i=1}^n T_i \sum_{\substack{H\subseteq\{1\ldots n-1\}\\ \#H=r-2}} (-1)^{A+H} f_H \sum_{\alpha \in H^c} (-1) \rho(\{\alpha\},H^c\sm {\alpha})x_{i\alpha}X_{A^c (H \cup \{\alpha\})^c}\\
   &=(-1)^{r} \Delta_A^{r-2} \sum_{i=1}^n T_i \sum_{\substack{H\subseteq\{1\ldots n-1\}\\ \#H=r-2}} (-1)^{A+H} f_H \sum_{\alpha \in H^c} \rho(\{\alpha\},H^c\sm {\alpha})x_{i\alpha}X_{A^c (H \cup \{\alpha\})^c}
\end{align*}
Now applying Lemma \ref{Lemma-ids} (1) we get 
\begin{align*}
    \phantom{\delta(\vp_{r-1}^{r-1}(e_A))} 
    &=(-1)^{r} \Delta_A^{r-2} \sum_{i=1}^n T_i \sum_{\substack{H\subseteq\{1\ldots n-1\}\\ \#H=r-2}} (-1)^{A+H} f_H Y_{A,H,i}\\
    &=(-1)^{r} \Delta_A^{r-2} \sum_{\substack{H\subseteq\{1\ldots n-1\}\\ \#H=r-2}}(-1)^{A+H} f_H\sum_{i=1}^n Y_{A,H,i} T_i\\
    &=(-1)^{r} \Delta_A^{r-2} \sum_{\substack{H\subseteq\{1\ldots n-1\}\\ \#H=r-2}}(-1)^{A+H} f_H\sum_{i\in A} Y_{A,H,i} T_i.\\
\end{align*}
Here the last equality follows from the fact that $Y_{A,H,i}=0 $ if $i\not\in A$.

Now for the other composition,

\begin{align*}
    \vp_{r-1}^{r-2}(d(e_A))&=\vp_{r-1}^{r-2}( \sum_{\beta\in A} \rho(\{\beta\}, A\sm \{\beta\}) \Delta_\beta^{r-1} e_{A\sm\{\beta\}})\\
    &= \sum_{\beta\in A} \rho(\{\beta\}, A\sm \{\beta\}) \Delta_\beta^{r-1} \vp_{r-1}^{r-2}(e_{A\sm\{\beta\}})\\
    &=\sum_{\beta\in A} \rho(\{\beta\}, A\sm \{\beta\}) \Delta_\beta^{r-1} (-1)^{r} \Delta_{A\sm \{\beta\}}^{r-2} \sum_{b\in A\sm \{\beta\}} \frac{T_b}{\Delta_b} \sum_{\substack{H\subseteq\{1\ldots n-1\} \\ \#H=r-2}} (-1)^{A\sm \{\beta\}+H} X_{(A\sm \{\beta\})^c H^c} f_H \\
    &=(-1)^r \sum_{\substack{H\subseteq\{1\ldots n-1\} \\ \#H=r-2}}  (-1)^{A+H} f_H \sum_{\beta\in A} \sum_{b\in A\sm \{\beta\}} (-1)^{\beta}\frac{T_b}{\Delta_b}\rho(\{\beta\}, A\sm \{\beta\}) \Delta_A^{r-2} \Delta_{\beta}  X_{(A\sm \{\beta\})^c H^c}\\
    &=(-1)^r \Delta_A^{r-2} \sum_{\substack{H\subseteq\{1\ldots n-1\} \\ \#H=r-2}}  (-1)^{A+H} f_H \sum_{i\in A} \frac{T_i}{\Delta_i} \sum_{\gamma \in A\sm \{i\}} (-1)^{\gamma} \rho(\{\gamma\}, A\sm \{\gamma\})\Delta_{\gamma} X_{(A\sm \{\gamma\})^c H^c}
\end{align*}
Now, set
\[\Theta :=(-1)^r \Delta_A^{r-2} \sum_{\substack{H\subseteq\{1\ldots n-1\} \\ \#H=r-2}}  (-1)^{A+H} f_H.\] So,
\begin{align*}
    \vp_{r-1}^{r-2}(d(e_A))
    &= \Theta \cdot  \sum_{i\in A} \frac{T_i}{\Delta_i} \sum_{\gamma \in A\sm \{i\}} (-1)^{\gamma} \Delta_{\gamma} \rho(\{\gamma\}, A\sm \{\gamma\}) X_{(A\sm \{\gamma\})^c H^c}.
\end{align*}
We now write \[ X_{(A\sm \{\gamma\})^c H^c}=\rho(\{\gamma\}, A^c \sm \{\gamma\}) Y_{A,H,\gamma}=\rho(\{\gamma\}, A^c \sm \{\gamma\}) \sum_{\alpha\in H^c} \rho(\{\alpha\},H^c\sm \{\alpha\}) x_{\gamma \alpha} X_{A^c (H\cup {\alpha})^c}.\]
Hence by Lemma \ref{Lemma-ids} (2) we have 

\begin{align*}
 \rho(\{\gamma\}, A\sm \{\gamma\})X_{(A\sm \{\gamma\})^c H^c}&=  \rho(\{\gamma\}, A\sm \{\gamma\})\rho(\{\gamma\}, A^c \sm \{\gamma\}) \sum_{\alpha\in H^c} \rho(\{\alpha\},H^c\sm \{\alpha\}) x_{\gamma \alpha} X_{A^c (H\cup {\alpha})^c}   \\
 &=(-1)^{\gamma-1} \sum_{\alpha\in H^c} \rho(\{\alpha\},H^c\sm \{\alpha\}) x_{\gamma \alpha} X_{A^c (H\cup {\alpha})^c}.
\end{align*}
So, returning to the original expression,
\begin{align*}
   \vp_{r-1}^{r-2}(d(e_A))&= \Theta \cdot  \sum_{i\in A} \frac{T_i}{\Delta_i} \sum_{\gamma \in A\sm \{i\}} (-1)^{\gamma} \Delta_\gamma  (-1)^{\gamma-1} \sum_{\alpha\in H^c} \rho(\{\alpha\},H^c\sm \{\alpha\}) x_{\gamma \alpha} X_{A^c (H\cup {\alpha})^c}\\
    &=\Theta \cdot  (-1)\sum_{i\in A} \frac{T_i}{\Delta_i} \sum_{\alpha\in H^c} \rho(\{\alpha\},H^c\sm \{\alpha\}) X_{A^c (H\cup {\alpha})^c}\sum_{\gamma \in A\sm \{i\}} \Delta_\gamma  x_{\gamma \alpha}\\
\end{align*}
Using the fact that $\sum_{i=1}^n \Delta_i x_{i\alpha}=0$ we get that $\sum_{\gamma \in A\sm \{i\}} \Delta_\gamma  x_{\gamma \alpha}=-\sum_{\gamma \in A^c \cup \{i\}} \Delta_\gamma  x_{\gamma \alpha} $. Therefore the previous line becomes
\begin{align*}
\phantom{\vp_{r-1}^{r-2}(d(e_A))} 
    &=\Theta \cdot  (-1)\sum_{i\in A} \frac{T_i}{\Delta_i} \sum_{\alpha\in H^c} \rho(\{\alpha\},H^c\sm \{\alpha\}) X_{A^c (H\cup {\alpha})^c}((-1)\sum_{\gamma \in A^c \cup \{i\}} \Delta_\gamma  x_{\gamma \alpha})\\  
    &=\Theta \cdot  \sum_{i\in A} \frac{T_i}{\Delta_i} \sum_{\gamma \in A^c \cup \{i\}} \Delta_\gamma \sum_{\alpha\in H^c} \rho(\{\alpha\},H^c\sm \{\alpha\})   x_{\gamma \alpha} X_{A^c (H\cup {\alpha})^c}\\
\end{align*}
Using Lemma \ref{Lemma-ids} (1), we see that $\sum_{\alpha\in H^c} \rho(\{\alpha\},H^c\sm \{\alpha\})   x_{\gamma \alpha} X_{A^c (H\cup {\alpha})^c}=Y_{A,H,\gamma}$. Thus, we have
\begin{align*}
\phantom{\vp_{r-1}^{r-2}(d(e_A))}   
&=\Theta \cdot  \sum_{i\in A} \frac{T_i}{\Delta_i} \sum_{\gamma \in A^c \cup \{i\}} \Delta_\gamma Y_{A,H,\gamma}.\\
\end{align*}
Finally, using that $Y_{A,H,\gamma}=0$ for $\gamma\in A^c$, the expression simplifies to 
\begin{align*}
\phantom{\vp_{r-1}^{r-2}(d(e_A))}   
&=\Theta \cdot  \sum_{i\in A} \frac{T_i}{\Delta_i} \Delta_i Y_{A,H,i}\\
&=(-1)^r \Delta_A^{r-2} \sum_{\substack{H\subseteq\{1\ldots n-1\} \\ \#H=r-2}}  (-1)^{A+H} f_H \sum_{i\in A} \frac{T_i}{\Delta_i} \Delta_i Y_{A,H,i}\\
&=(-1)^r \Delta_A^{r-2} \sum_{\substack{H\subseteq\{1\ldots n-1\} \\ \#H=r-2}}  (-1)^{A+H} f_H \sum_{i\in A}  Y_{A,H,i}T_i.
\end{align*}
We have shown that 
\[\vp_{r-1}^{r-2}(d(e_A))=(-1)^r \Delta_A^{r-2} \sum_{\substack{H\subseteq\{1\ldots n-1\} \\ \#H=r-2}}  (-1)^{A+H} f_H \sum_{i\in A}  Y_{A,H,i} T_i  \]
and 
\[\delta(\vp_{r-1}^{r-1}(e_A))
    =(-1)^{r} \Delta_A^{r-2} \sum_{\substack{H\subseteq\{1\ldots n-1\}\\ \#H=r-2}}(-1)^{A+H} f_H\sum_{i\in A} Y_{A,H,i} T_i\]
    so the commutativity of the diagram is proven.

\end{proof}

\begin{Cor}\label{Cor-extended base case} Suppose $\vp_{r-1}^{r-1}$ and $\vp_{r-1}^{r-2}$ are the maps defined in Theorem \ref{Th-base case}. Consider the following two squares of Diagram \eqref{Diagram-lift} where $t=r-1$.

\[\begin{tikzcd}
0=\bigwedge^{r} (0)^{n-1} \ar{r} &\bigwedge^{r-1} (R1)^{n-1} \ar{r}{\delta} & \bigwedge^{r-2} (\bigoplus RT_i)^{n-1} \\
\bigwedge^{r+1} R^n \ar{r}{d^{r+1}_{r-1}}\ar{u}{0} &\bigwedge^r R^n \ar{r}{d^r_{r-1}} \ar{u}{\vp_{r-1}^{r-1}}& \bigwedge^{r-1} R^n \ar{u}{\vp_{r-1}^{r-2}}
\end{tikzcd}\]
This diagram commutes.
\end{Cor}
\begin{proof}
This follows from the injectivity of $\delta$, since the the top row is the tail of a resolution of $I^{r-1}$, and Theorem \ref{Th-base case}: We have that $\im(\vp_{r-1}^{r-1} \circ d_{r-1}^{r+1}) \subseteq \ker \delta =0$.
\end{proof}

\begin{Not}
Let $s_e(y_1,\ldots, y_d)$ be the complete homogeneous symmetric function of degree $e$ in $y_1,\ldots, y_d$. For $A=\{a_1,\ldots, a_d\}\subseteq \{1,\ldots, n\}$ define $h_e(A)=s_e(\frac{T_{a_1}}{\Delta_{a_1}},\cdots,\frac{T_{a_d}}{\Delta_{a_d}}).$ 
\end{Not}

\begin{Th}\label{Th-full lift}
For $r>1$ let $\vp^{r-1}_{r-1}$ be the maps defined in Theorem \ref{Th-base case} and let $\vp_0^0:\bigwedge^0 R^n \to [\bigwedge^{0} S^{n-1} ]_{0}$ be the map $\vp_0^0(1)=1$. Then for all $t,r\geq 1$ define functions $\vp^{r-1}_t: \bigwedge^r R^n \to [\bigwedge^{r-1} S^{n-1} ]_{t-r+1}$ as follows:
\[ \vp^{r-1}_t(e_A):= \begin{cases} 
      0 & t< r-1, \\
      \vp^{r-1}_{r-1}(e_A) & t=r-1, \\
      \vp^{r-1}_{r-1}(e_A)(\Delta_A^{t-r+1} h_{t-r+1}(A)) & t>r-1. 
   \end{cases}
\]
(Note that this is the same definition of $\vp_{r-1}^{r-2}$ as in Theorem \ref{Th-base case}). Then 
\[ \begin{tikzcd}
\cdots \ar{r}&\left[\bigwedge^{r-1} S^{n-1} \right]_{t-r+1}\ar{r}{\delta^{r-1}} & \cdots \ar{r}& \left[\bigwedge^1 S^{n-1} \right]_{t-1}\ar{r}{\delta^{1}} & \left[S\right]_t\ar{r}{\delta^{0}}& I^t \\
\cdots \ar{r}&\bigwedge^{r} R^n \ar{r}{d_t^r} \ar{u}{\vp_t^{r-1}}& \cdots \ar{r}& \bigwedge^2 R^n \ar{r}{d_t^2} \ar{u}{\vp_t^1}& \bigwedge^1 R^n \ar{r}{d_t^1} \ar{u}{\vp_t^0}& (\Delta_1^t,\ldots,\Delta_n^t) \ar[hookrightarrow]{u} \\
\end{tikzcd}
\]
commutes, where the rightmost map is the natural inclusion, the top row is the $t$-th strand of the $K_\bullet([F_1\ldots F_{n-1}];S)$ and the bottom row is $K_\bullet([\Delta_1^t\ldots \Delta_n^t];R)$.
\end{Th}

Again, before proving this theorem lets return to Example \ref{ex-lift1}.

\begin{Ex}\label{ex-lift2} Suppose $n=3$ and $t=2$ then Diagram \eqref{Diagram-lift} becomes:
\[\begin{tikzcd}
  0\ar{r}& \left[\bigwedge^2 S^2 \right]_0 \ar{r}& \left[\bigwedge^1 S^2 \right]_1 \ar{r}& \left[\bigwedge^0 S^{2}\right]_2 \ar{r} & I^2 \ar{r} & 0\\
  0 \ar{r}& \bigwedge^3 R^3 \ar{r} \ar{u}{\vp^2_2}& \bigwedge^3 R^2 \ar{r} \ar{u}{\vp^1_2}& \bigwedge^1 R^3 \ar{r} \ar{u}{\vp^0_2}& (\Delta_1^2,\Delta_2^2,\Delta_3^2) \ar[hookrightarrow]{u} \ar{r}&0. 
\end{tikzcd}
\]
We saw in Example \ref{ex-lift1} that
\begin{align*}
    \vp_2^2(e_1 \wedge e_2 \wedge e_3)&= (-1)^2 \Delta_1  \Delta_2 \Delta_3( (-1)^{(6+3)}f_1 \wedge f_2)\\
    &=- \Delta_1  \Delta_2 \Delta_3(f_1 \wedge f_2).
\end{align*}  
Now using Theorem \ref{Th-full lift} we compute $\vp_2^0$ and $\vp_2^1$: $\vp_0^0(e_a)=1$ and $h_2(\{a\})=\frac{T_a^2}{\Delta_a^2}$, so 
\[\vp_2^0(e_a)=(1)(\Delta_a^2)\left(\frac{T_a^2}{\Delta_a^2} \right)=T_a^2.\]
For $\vp_2^1$ we first need to compute $\vp_1^1$:
\begin{align*}
    \vp_1^1(e_a\wedge e_b)&=(-1)^{2-1} (\Delta_a \Delta_b)^{2-2}((-1)^{a+b+1} X_{\{a,b\}^c,1}f_1 +(-1)^{a+b+1} X_{\{a,b\}^c,1}f_2)\\
    &=\begin{cases}  -( x_{3,2} f_1 -x_{3,1} f_2 )& (a,b)=(1,2)\\
    - (-x_{2,2} f_1 + x_{2,1} f_2)& (a,b)=(1,3)\\
    -( x_{1,2} f_1  -x_{1,1} f_2)& (a,b)=(2,3)\\
    \end{cases}.
\end{align*}
Now $h_{1}(\{a,b\})=\frac{T_a}{\Delta_a}+\frac{T_b}{\Delta_b}$ and we have,
\begin{align*}
    \vp_2^1(e_a\wedge e_b)&= \vp_1^1(e_a \wedge e_b)(\Delta_a \Delta_b)\left( \frac{T_a}{\Delta_a}+\frac{T_b}{\Delta_b}\right)\\
    &=\vp_1^1(e_a \wedge e_b)(\Delta_b T_a+ \Delta_a T_b)\\
    &=\begin{cases}  -(\Delta_2 T_1+ \Delta_1 T_2)( x_{3,2} f_1 -x_{3,1} f_2 )& (a,b)=(1,2)\\
    -(\Delta_3 T_1+ \Delta_1 T_3) (-x_{2,2} f_1 + x_{2,1} f_2)& (a,b)=(1,3)\\
    -(\Delta_3 T_2+ \Delta_2 T_3)( x_{1,2} f_1  -x_{1,1} f_2)& (a,b)=(2,3)\\
    \end{cases},
\end{align*}
which agrees with the computation in Example \ref{ex-lift1}.
\end{Ex}

\begin{proof}[Proof of Theorem \ref{Th-full lift}]
The commutativity of the rightmost square is immediate so we are done once we show that for all $r\geq 2$ the following square commutes:
\[\begin{tikzcd}
\left[ \bigwedge^{r-1} S^{n-1} \right]_{t-r+1} \ar{r}{\delta^{r-1}} & \left[ \bigwedge^{r-2} S^{n-1} \right]_{t-r+2} \\
\bigwedge^r R^n \ar{r}{d^r_t} \ar{u}{\vp_{t}^{r-1}}& \bigwedge^{r-1} R^n \ar{u}{\vp_{t}^{r-2}}
\end{tikzcd}.\]
If $t<r-2$ the top row vanishes and both vertical maps are zero, so commutativity is clear. The case that $t=r-2$ is addressed by Corollary \ref{Cor-extended base case}. Finally, the case where $t=r-1$ is handled by Theorem \ref{Th-base case}. So it is left to check the cases $t>r-1$.\\
Before computing the two compositions we note the following key identity. Let $A\subseteq \{1,\ldots, n\}$ with $\#A=r$. Using Corollary \ref{Cor-extended base case} we have that: 
\begin{equation}\label{key id}
    \begin{aligned}
    0&=\vp_{r-2}^{r-2}(d_{r-2}^{r})(e_A)\\
    &=\vp_{r-2}^{r-2}(\sum_{\alpha\in A} \rho(\{\alpha\},A\sm \{\alpha\})\Delta_\alpha^{r-2} e_{A\sm \alpha} )\\
    &=\sum_{\alpha\in A} \rho(\{\alpha\},A\sm \{\alpha\})\Delta_\alpha^{r-2} \vp_{r-2}^{r-2}( e_{A\sm \alpha} ).
    \end{aligned}
\end{equation}

Now we are ready to show that the square commutes. Let $A\subseteq \{1\ldots n\}$ with $\#A=r$ and set $e=t-r+1$. Then,
\begin{align*}
    \delta^{r-1}(\vp_t^{r-1}(e_A)) &= \delta^{r-1}(\vp^{r-1}_{r-1}(e_A)(\Delta_A^{e} h_{e}(A))) \\
    &=\Delta_A^{e} h_{e}(A)\delta^{r-1}(\vp^{r-1}_{r-1}(e_A)),
\end{align*}
where the second equality follows from the $S$-linearity of $\delta$. Now using Theorem \ref{Th-base case} we have,
\begin{align*}
    \phantom{\delta^{r-1}(\vp_t^{r-1}(e_A))} &= \Delta_A^{e} h_{e}(A) \vp_{r-1}^{r-2}(d_{r-1}^r(e_A))\\
    &=\Delta_A^{e} h_{e}(A) \vp_{r-1}^{r-2}(\sum_{\alpha\in A} \rho(\{\alpha\},A\sm\{\alpha\}) \Delta_\alpha^{r-1} e_{A\sm \{\alpha\}})\\
    &=\Delta_A^{e} h_{e}(A) \sum_{\alpha\in A}\rho(\{\alpha\},A\sm\{\alpha\}) \Delta_\alpha^{r-1} \vp_{r-1}^{r-2}( e_{A\sm \{\alpha\}})\\
    &=\Delta_A^{e} h_{e}(A) \sum_{\alpha\in A}\rho(\{\alpha\},A\sm\{\alpha\}) \Delta_\alpha^{r-1} (\vp_{r-2}^{r-2}( e_{A\sm \{\alpha\}})\frac{\Delta_A}{\Delta_\alpha} h_1(A\sm\{\alpha\}))\\
    &=\Delta_A^{e+1} h_{e}(A) \sum_{\alpha\in A}\rho(\{\alpha\},A\sm\{\alpha\}) \Delta_\alpha^{r-2} \vp_{r-2}^{r-2}( e_{A\sm \{\alpha\}}) h_1(A\sm\{\alpha\})\\
    &=\Delta_A^{e+1} \sum_{\alpha\in A}\rho(\{\alpha\},A\sm\{\alpha\}) \Delta_\alpha^{r-2} \vp_{r-2}^{r-2}( e_{A\sm \{\alpha\}}) \sum_{\beta\in A\sm\{\alpha\}} \frac{T_\beta}{\Delta_\beta} h_e(A).
\end{align*}
Now apply the fact that for $\beta \in A$, $h_{e+1}(A)=\frac{T_\beta}h_e(A)+h_{e+1}(A\sm \{\beta\})$ to see that the above is 
\begin{align*}
     \phantom{\delta^{r-1}(\vp_t^{r-1}(e_A))} &=\Delta_A^{e+1} \sum_{\alpha\in A}\rho(\{\alpha\},A\sm\{\alpha\}) \Delta_\alpha^{r-2} \vp_{r-2}^{r-2}( e_{A\sm \{\alpha\}}) \sum_{\beta\in A\sm\{\alpha\}} (h_{e+1}(A)-h_{e+1}(A\sm \{\beta\}))\\
     &=\Delta_A^{e+1} \sum_{\alpha\in A}\rho(\{\alpha\},A\sm\{\alpha\}) \Delta_\alpha^{r-2} \vp_{r-2}^{r-2}( e_{A\sm \{\alpha\}}) \sum_{\beta\in A\sm\{\alpha\}} h_{e+1}(A)\\
     &\qquad \qquad -\Delta_A^{e+1} \sum_{\alpha\in A}\rho(\{\alpha\},A\sm\{\alpha\}) \Delta_\alpha^{r-2} \vp_{r-2}^{r-2}( e_{A\sm \{\alpha\}}) \sum_{\beta\in A\sm\{\alpha\}} h_{e+1}(A\sm \{\beta\}))
\end{align*}
But $\sum_{\beta\in A\sm\{\alpha\}} h_{e+1}(A)=(\#A-1)h_{e+1}(A)=(r-1)h_{e+1}(A)$. So by identity \eqref{key id}, 
\begin{align*}
    0&=\Delta_A^{e+1} \sum_{\alpha\in A}\rho(\{\alpha\},A\sm\{\alpha\}) \Delta_\alpha^{r-2} \vp_{r-2}^{r-2}( e_{A\sm \{\alpha\}}) \sum_{\beta\in A\sm\{\alpha\}} h_{e+1}(A)\\
    &=(r-1)\Delta_A^{e+1}  h_{e+1}(A)\sum_{\alpha\in A}\rho(\{\alpha\},A\sm\{\alpha\}) \Delta_\alpha^{r-2} \vp_{r-2}^{r-2}( e_{A\sm \{\alpha\}}).
\end{align*}

Hence we have
\begin{align*}
    \phantom{\delta^{r-1}(\vp_t^{r-1}(e_A))}&=-\Delta_A^{e+1} \sum_{\alpha\in A}\rho(\{\alpha\},A\sm\{\alpha\}) \Delta_\alpha^{r-2} \vp_{r-2}^{r-2}( e_{A\sm \{\alpha\}}) \sum_{\beta\in A\sm\{\alpha\}} h_{e+1}(A\sm \{\beta\}))\\
    &=-\Delta_A^{e+1} \sum_{\alpha\in A} h_{e+1}(A\sm \alpha) \sum_{\beta\in A\sm\{\alpha\}} 
    \rho(\{\beta\},A\sm\{\beta\}) \Delta_\beta^{r-2} \vp_{r-2}^{r-2}( e_{A\sm \{\beta\}}).
\end{align*}
By identity \eqref{key id}, we see that $0=\sum_{\beta\in A\sm\{\alpha\}} 
    \rho(\{\beta\},A\sm\{\beta\}) \Delta_\beta^{r-2} \vp_{r-2}^{r-2}( e_{A\sm \{\beta\}})+\rho(\{\alpha\},A\sm\{\alpha\}) \Delta_\alpha^{r-2} \vp_{r-2}^{r-2}( e_{A\sm \{\alpha\}})$. So the above expression is equal to
\begin{align*}
    \phantom{\delta^{r-1}(\vp_t^{r-1}(e_A))}&=-\Delta_A^{e+1} \sum_{\alpha\in A} h_{e+1}(A\sm \alpha) (-\rho(\{\alpha\},A\sm\{\alpha\}) \Delta_\alpha^{r-2} \vp_{r-2}^{r-2}( e_{A\sm \{\alpha\}}))\\
    &= \sum_{\alpha\in A}  \rho(\{\alpha\},A\sm\{\alpha\}) \Delta_\alpha^{r-2} \vp_{r-2}^{r-2}( e_{A\sm \{\alpha\}}) \Delta_A^{e+1} h_{e+1}(A\sm \alpha)\\
    &=\sum_{\alpha\in A}  \rho(\{\alpha\},A\sm\{\alpha\}) \Delta_\alpha^{r-2} \vp_{r-2}^{r-2}( e_{A\sm \{\alpha\}})  \left(\frac{\Delta_A}{\Delta_\alpha}\right)^{e+1}  h_{e+1}(A\sm \alpha)\\
    &=\sum_{\alpha\in A}  \rho(\{\alpha\},A\sm\{\alpha\}) \Delta_\alpha^{r-2+e+1} \vp_{r-2}^{r-2}( e_{A\sm \{\alpha\}}) \Delta_{A\sm \alpha} ^{e+1}  h_{e+1}(A\sm \alpha)\\
    &=\sum_{\alpha\in A}  \rho(\{\alpha\},A\sm\{\alpha\}) \Delta_\alpha^{t} \vp_t^{r-2}(e_{A\sm \{\alpha \}})\\
    &=\vp_t^{r-2}(\sum_{\alpha\in A}  \rho(\{\alpha\},A\sm\{\alpha\}) \Delta_\alpha^{t} e_{A\sm \{\alpha \}})\\
    &=\vp_t^{r-2}(d_t^r(e_A)).
\end{align*}
\end{proof}

The results of this section are highly specialized to the case that $X$ is size $n\times(n-1)$, in all other cases the Rees algebra of the ideal of maximal minors is substantially less nice and it is much more difficult to access a resolution of $I^t$, cf. \cite{Akin-Buchsbaum-Weyman81Resolutionsofdeterminantalideals:Thesubmaximalminors}. However, this is the only specialized aspect of this argument. Due to the elementary computational nature of the proof, Theorem \ref{Th-full lift} holds for any grade 2 perfect ideal of linear type with mild assumptions assumptions on the ambient ring. 


\section{The $n\times (n-1)$ Case}\label{section-n times n-1 full section }
For this section let $X$ be a $n\times (n-1)$ matrix of indeterminates, $R=\CC[X]$ and $I=I_{n-1}(X)$. We write $d_i$ for the determinant of the matrix obtained by deleting the $i$-th row of $X$. As noted in Section \ref{section-Gl invariant ideals} $R\cong \Sym(F\otimes G)$ where $F=\CC^n$, $G=\CC^{n-1}$ and $\Gl=\Gl_n\times \Gl_{n-1}$ acts on $R$.

\subsection{The Cyclic Local Cohomology Module}\label{section-spec to n times (n-1)}
By Proposition \ref{prop-cyclic R module, general case}, we have that $H_\maxm^{(n-1)^2-1}(R/I^t)= H_\maxm^{n^2-2n}(R/I^t)$ is a cyclic $R$-module. Define $J_t$ to be the ideal such that

\[H_\maxm^{n^2-2n}(R/I^t)\cong R/J_t.\]

We will utilize the lift constructed in Section \ref{section-lift} to describe the modules $\Ext_R ^n(R/I^t,R)$ as submodules of $H_I^n(R)$. After constructing an isomorphism of $D$-modules $H_I^n(R)\to H_m^{n(n-1)}(R)$ we obtain a description of $\Ext_R ^n(R/I^t,R)$ as a submodule of $H_m^{n(n-1)}(R)$ which we can use to directly compute $\ann_R \Ext_R ^n(R/I^t,R)=J_t$.

\subsection{Description of $\Ext_R ^n(R/I^t,R)$}\label{subsetction-description of EXT}
Let $\Delta_i=(-1)^i d_i$ and $S=\Sym(R^{n-1})$. Then we have the following commutative diagram:

\begin{equation}\label{Diagram-top cohom desription}
    \begin{tikzcd}
 R\left[\frac{1}{\Delta_1},\ldots, \frac{1}{\Delta_n}\right]\ar{r}& H^n( \check{C}^\bullet(\Delta_1,\ldots, \Delta_n; R))\cong  H_I^n(R)\ar{r}& 0\\
R \ar{r} \ar{u}{\alpha \mapsto \frac{\alpha}{(\prod \Delta_i)^t}} & \ar{u}H^n(K^\bullet (\Delta_1^t,\ldots, \Delta_n^t;R)) \ar{r} & 0\\
&\Ext_R^n(R/I^t,R) \ar{u}{\psi_t} & 
 \end{tikzcd}
\end{equation}
By Corollary \ref{Cor-union}, the composition of vertical maps on the right is injective. Moreover $\psi_t$ is induced by the map $\vp_t^{n-1}: R\cong \bigwedge^n R^n\to [\bigwedge^{n-1} (S)^{n-1}]_{t-n}\cong [S]_{t-n+1}$ described in Theorem \ref{Th-base case} and Theorem \ref{Th-full lift}. This map is zero for $t<n-1$. For $t=n-1$, we have $\vp_{n-1}^{n-1}$, and hence $\psi_{n-1}$ is multiplication by the constant:
\[(-1)^{n-1} (\prod_{i=1}^n \Delta_i)^{n-2}.\]
Thus, the image of $\Ext_R^n(R/I^{n-1},R)$ is generated by $\frac{(\prod_{i=1}^n \Delta_i)^{n-2}}{(\prod_{i=1}^n \Delta_i)^{n-1}}=\frac{1}{\prod_{i=1}^n \Delta_i}$ in $H_I^n(R)$. For $t\geq n$ we see that for $|\alpha|=t-n+1$,
\[\psi_t(T^\alpha)= (-1)^{n-1} (\prod_{i=1}^n \Delta_i)^{n-2} 
(\prod_{i=1}^n \Delta_i)^{t-n+1}\frac{1}{\Delta^\alpha}=(-1)^{n-1} (\prod_{i=1}^n \Delta_i)^{t-1} \frac{1}{\Delta^\alpha}.\]
Since $d_i$ and $\Delta_i$ agree up to sign the above discussion proves the following:
\begin{Th}\label{Th-Ext gens}
Under the embedding $\Ext_R^n(R/I^t,R) \hookrightarrow H_I^n(R) $ of Diagram \eqref{Diagram-top cohom desription}, $\Ext_R^n(R/I^t,R)$ is the submodule generated by $$\left\{\frac{1}{\prod_{i=1}^n d_i} \cdot  \frac{1}{d^\alpha} \right\}_{|\alpha|=t-n+1}.$$ 
\end{Th}

Recall that $H_I^n(R)$ is a cyclic $\DD$-module. The following result allows us to describe the images of the modules $\Ext_R ^n(R/I^t,R)$ in $H_I^n(R)$ in a manner related to the $\DD$-module structure of $H_I^n(R)$.

\begin{Prop}\label{prop-Cayley type}\cite[Remark 3.8]{Lorincz-Raicu-Walther-Weyman17Bernstein-Satopolynomialsformaximalminorsandsub-maximalPfaffians}
\cite{Lorincz17Theb-functionsofquiversemi-invariants}
Let $\underline s=(s_1,\ldots, s_n)$ and $s=\sum s_i$. For each $i$, we have 
\[d_i^* \bullet (d_i\cdot d^{\underline s})=(s_i+1)(s+2)(s+3)\cdots(s+n)d^{\underline s}\]
\end{Prop}
This proposition immediately gives us the following.

\begin{Prop}\label{prop-Ext diff ops}
Under the embedding induced by Diagram \eqref{Diagram-top cohom desription}, for $t\geq n-1$, we have
\[\Ext_R ^n(R/I^t,R)=\sum_{|\alpha|=t-n+1} R \cdot (d^\alpha)^* \bullet \frac{1}{\prod_{i=1}^n d_i} . \]
\end{Prop}

By Theorem \ref{Th-LSW} the $\DD$-modules $H_I^n(R)$ and $H_\maxm^{n(n-1)}(R)$ are isomorphic cyclic $D$-modules. To describe a $\DD$-isomorphism between them it is sufficient to choose a socle generator of $H_I^n(R)$ and of $H_\maxm^{n(n-1)}(R)$. Choose 
\[\frac{1}{\prod_{i=1}^n d_i} \in  \Soc (H_I^n(R))\] and
\[\frac{1}{\underline{x}}:=\frac{1}{\prod_{ij} x_{ij}}\in \Soc (H_\maxm^{n(n-1)}(R)).\] We observe the image of $\Ext_R ^n(R/I^t,R)$ in $H_\maxm^{n(n-1)}(R)$ under the map induced by $\frac{1}{\prod_{i=1}^n d_i} \mapsto \frac{1}{\underline{x}}$.

\begin{Prop}\label{prop-Ext in H_m}
For $t\geq n-1$, we have
\[\Ext_R ^n(R/I^t,R)\cong \sum_{|\alpha|=t-n+1} R \cdot (d^\alpha)^* \bullet \frac{1}{\underline{x}}, \]
where we write $\frac{1}{\underline{x}}$ for the class in $H_\maxm^{n(n-1)}(R)$.
\end{Prop}

\begin{Ex}
Let $n=t=3$. Then,
\[\Ext_R ^3(R/I^3,R)\cong \sum_{i=1}^3 R \cdot (d_i)^* \bullet \frac{1}{x_{1,1}x_{1,2}x_{2,1}x_{2,2}x_{3,1}x_{3,2}}.\]
Thus, $\Ext_R ^3(R/I^3,R)\subseteq H_\maxm^{n(n-1)}(R)$ is generated as an $R$-module by
\[
    \frac{1}{\underline{x}}\left( \frac{1}{x_{2,1}x_{3,2}}- \frac{1}{x_{2,2}x_{3,1}} \right),\]
    \[
        \frac{1}{\underline{x}}\left( \frac{1}{x_{1,1}x_{3,2}}- \frac{1}{x_{1,2}x_{3,1}} \right),\]
        and
            \[\frac{1}{\underline{x}}\left( \frac{1}{x_{1,1}x_{2,2}}- \frac{1}{x_{1,2}x_{2,1}} \right).\]
\end{Ex}

Using this description $\Ext_R ^n(R/I^t,R)$, we can utilize the $\DD$-module structure of $H_\maxm^{n(n-1)}(R)$ to describe the annihilator of $\Ext_R ^n(R/I^t,R)$. Recall from Section \ref{section-equivariant pairing} that $R^*=\CC[\{\partial_{ij}\}]$ and for a polynomial $f\in R$ we write $f^*=f(\{\partial_{ij}\})\in R^*$. For an element $f\in R$ we can form the $R^*$ module generated by $f$, where $R^*$ acts by differentiation.

\begin{Prop}\label{prop-ann of Ext as R^*}
Let $t\geq n-1$. Then 
\[(\ann_R \Ext_R ^n(R/I^t,R))^*=\ann_{R^*} \sum_{|\alpha|=t-n+1} R^* \cdot d^\alpha.\]
\end{Prop}
\begin{proof}
Let $\zeta=\frac{1}{\underline x}\in H_\maxm^{n(n-1)}(R)$ then $H_\maxm^{n(n-1)}(R)=\DD \cdot\zeta$ and $\ann_\DD \zeta= \DD\cdot\maxm$. Now $f\in \ann_R \Ext_R ^n(R/I^t,R)$ if and only if for all $|\alpha|=t-n+1$ we have $f {d^\alpha}^* \bullet \zeta =0 $. Now $f {d^\alpha}^* \bullet \zeta =0$ if and only if $f {d^\alpha}^* \in \DD\cdot\maxm$. So, applying the Fourier transform which sends $x_{ij}\mapsto \partial_{ij}$ ,$\partial_{ij}\mapsto -x_{ij}$, we have that $f {d^\alpha}^* \in \DD\cdot\maxm$ if and only if $f^* d^\alpha \in \DD\cdot (\maxm^*)$ if and only if $f^* \bullet d^\alpha =0$. Hence $f\in \ann_R \Ext_R ^n(R/I^t,R)$ if and only if $f^* \in \ann_{R^*} \sum_{|\alpha|=t-n+1} R^* \cdot d^\alpha$.
\end{proof}

\subsection{The annihilator of $\Ext_R ^n(R/I^t,R)$}
Recall from Section \ref{section-equivariant pairing} that for all $k\geq 0$ there exists a $\Gl$-equivariant pairing $\langle \;,\; \rangle: [R^*]_k \times [R]_k \to \CC$ induced by differentiation.

\begin{Prop}\label{prop-GL orbit kills minors}
Let $k\geq 1$, $\lambda=(k+1)$ and $N=[I^k]_{(n-1)k}=\sum_{|\alpha|=k}  \CC \cdot d^\alpha$. Then for all $f$ in the $\Gl$-orbit of $\det_\lambda$, $f^*\bullet N=0$.
\end{Prop}
\begin{proof}
$\det_\lambda = x_{1,1}^{t+1}$ so for all $|\alpha|=t$ we have that $(\det_\lambda)^* \bullet d^\alpha =0$. The claim then follows from Lemma \ref{lemma-orbit kils equivariant subspace}.
\end{proof}
We are now ready to prove Theorem \ref{Th-J_t description} in the $n\times (n-1)$ case. 

\begin{Th}\label{Th-J_t description nx(n-1)} If $t\leq n-2$ then $J_t=R$, for $t\geq n-1$,
\[\ann_R \Ext_R ^n(R/I^t,R)=J_t=I_{(t-n+2)}.\]
\end{Th}
\begin{proof}
In the case that $t\leq n-2$ we have that $\projdim_R(R/I^t)<n$ so clearly 
\[\ann_R \Ext_R ^n(R/I^t,R)=R.\] 
For $t\geq n-1$, first, we claim that 
\[I_{(t-n+2)}\subseteq \ann_R \Ext_R ^n(R/I^t,R).\] Let $f\in I_{(t-n+2)}$, then by Proposition \ref{prop-GL orbit kills minors}, $f^*\bullet d^\alpha=0$ for all $|\alpha|=t-n+1$. Thus $f^*\in \ann_{R^*} \sum_{|\alpha|=t-n+1} R^* \cdot d^\alpha$, so by Proposition \ref{prop-ann of Ext as R^*}, $f\in \ann_R \Ext_R ^n(R/I^t,R).$

Now for the other inclusions we note that $\Ext_R ^n(R/I^t,R)$ is $\Gl$-equivariant hence $\ann_R \Ext_R ^n(R/I^t,R)$ is a $\Gl$-invariant ideal. As was noted in Subsection \ref{section-Gl invariant ideals}, \cite{deConcini-Eisenbud-Procesi80Youngdiagramsanddeterminantalvarieties} proved that every $\Gl$-invariant ideal is of the form $I_\chi=\sum_{\lambda\in \chi} I_\lambda$ for some finite collection of incomparable partitions $\chi$.

Suppose for the sake of contradiction that $I_{(t-n+2)}\subsetneq \ann_R \Ext_R ^n(R/I^t,R)$ and set $I_\chi=\ann_R \Ext_R ^n(R/I^t,R)$ where $\chi$ is a collection of incomparable partitions. Thus there exists a partition $\mu\in \chi$ such that either $(t-n+2)> \mu$ or $(t-n+2)$ is incomparable to $\mu$. In either case we have that $((t-n+1)^{n-1})\geq \mu$, hence $I_{((t-n+1)^{n-1})}\subseteq \ann_R \Ext_R ^n(R/I^t,R)$. In particular this implies that

\[{\det}_{((t-n+1)^{n-1})}=d_{n}^{t-n+1}\in \ann_R \Ext_R ^n(R/I^t,R).\]
However, this is a contradiction because by Theorem \ref{Th-Ext gens} we have that $\frac{1}{\prod_{i=1}^n d_i} \frac{1}{d_n^{t-n+1}}\in \Ext_R ^n(R/I^t,R)$ but 
\[d_{n}^{t-n+1} \cdot \left(\frac{1}{\prod_{i=1}^n d_i} \frac{1}{d_n^{t-n+1}}\right)=\frac{1}{\prod_{i=1}^n d_i}\neq 0.\]

\end{proof}

In the next section we will generalize Theorem \ref{Th-J_t description nx(n-1)} to maximal minors of arbitrary size matrices using graded duality and results from \cite{Raicu-Weyman-Witt14Localcohomologywithsupportinidealsofmaximalminorsandsub-maximalPfaffians}. However, this approach does not recover a more general version of Theorem \ref{Th-Ext gens}, the description of the $\Ext$ modules embedded in local cohomology, which would seem to be much more subtle. For $X$ an arbitrary size $m\times n$ matrix, we have that $H_{I_n(X)}^{mn-n^2+1}(R)$ is not just a cokernel of a map of the \v{C}ech complex on the maximal minors of $X$. Hence, writing down a description of a socle generator or even a non-zero element of $H_{I_n(X)}^{mn-n^2+1}(R)$ is non-trivial. This makes even conjecturing an explicit description of the submodules $\Ext_R^{mn-n^2+1}(R/I_n(X)^t,R)\subset  H_{I_n(X)}^{mn-n^2+1}(R)\cong H^{mn}_\maxm(R)$ challenging.

\section{The General Case}\label{section-general case}
We return to the setting of Section \ref{section-Gl invariant ideals}: Let $F=\CC^m$ and $G=\CC^n$ where $m\geq n$. Then 
\[R:=\Sym(F\otimes G)=\CC[\{x_{ij}\}]=\CC[X] \text{ and }\Gl:= \Gl(F)\times \Gl(G).\]
Fix $I$ to be the ideal of $n\times n$ minors of $X$.

\begin{Th}\label{Th-J_t description}
Let $R$, $I$ be as above and set $\maxm$ to be the homogeneous maximal ideal. Then
\[H_\maxm^{n^2-1}(R/I^t)\cong R/J_t,\]
where $J_t=R$ for $t<n$, and for $t\geq n$, $ J_t=I_{(t-n+1)}$, i.e., the ideal generated by the $\Gl$ orbit of $x_{11}^{t-n+1}$.
\end{Th}

\begin{proof}
By graded duality we have the following isomorphism:
\[
    H^{n^2-1}_\maxm(R/I^t) \cong \Hom_R( \Ext^{mn-n^2+1}_R(R/I^t,R), H_\maxm^{mn}(R)). \\
\]
The $\Gl$ structure of $H_\maxm^{mn}(R)$ is given by,
\[H_\maxm^{mn}(R)=\bigoplus_{\substack{\lambda\in \ZZ^n_{dom}\\ \lambda_1 \leq -m}} S_{\lambda(n)} F\otimes S_\lambda G, \] 
where $S_{\lambda(n)} F\otimes S_\lambda G$ lives in degree $|\lambda|$. 
We begin describing the $\Gl$ structure of $H^{n^2-1}_\maxm(R/I^t)$ by first analyzing a single graded component.

\begin{align*}
    [H^{n^2-1}_\maxm(R/I^t)]_r &= [\Hom_R( \Ext^{mn-n^2+1}_R(R/I^t,R), H_\maxm^{mn}(R))]_r \\
    &= \Hom_\CC( [\Ext^{mn-n^2+1}_R(R/I^t,R)]_{-mn-r} , [H_\maxm^{mn}(R)]_{-mn})  \\
    &= \Hom_\CC( [\Ext^{mn-n^2+1}_R(R/I^t,R)]_{-mn-r} , (\bigwedge^m F)^{-n}\otimes (\bigwedge^n G)^{-m})\\
    &=\Hom_\CC( [\Ext^{mn-n^2+1}_R(R/I^t,R)]_{-mn-r} ,\CC) \otimes(\bigwedge^m F)^{-n}\otimes (\bigwedge^n G)^{-m}.
\end{align*}
Now by Theorem \ref{Th-RWW Ext graded pieces} we have that  
\[\Ext^{mn-n^2+1}_R(R/I^t,R)]_{-mn-r}= \bigoplus_{\lambda \in A(r)} S_{\lambda(n)} F \otimes S_\lambda G,\]
where 
\[A(r)=\{ \lambda\in \ZZ^n | \sum_{i=1}^n \lambda_i = -mn -r\text{ and } -m \geq \lambda_1\geq \cdots \geq \lambda_n \geq -t-(m-n) \}.\]
Dualizing into $\CC$ we get that 
\[\Hom_\CC( [\Ext^{mn-n^2+1}_R(R/I^t,R)]_{-mn-r} ,\CC) =  \bigoplus_{\lambda \in A(r)} S_{\lambda+(-m^n)+(n^m)} F \otimes S_\lambda G,\]
where 
\[B(r)=\{ \lambda\in \ZZ^n | \sum_{i=1}^n \lambda_i = mn+r\text{ and } t+(m-n)\geq \lambda_1\geq \cdots \geq \lambda_n \geq m \}.\]
With this we can now describe the decomposition of $H^{n^2-1}_\maxm(R/I^t)$ into irreducible $\Gl$-representations: 
\begin{align*}
    [H^{n^2-1}_\maxm(R/I^t)]_r &=  \bigoplus_{\lambda \in B(r)} S_{\lambda+(-m^n)+(n^m) }F \otimes S_\lambda G \otimes(\bigwedge^m F)^{-n}\otimes (\bigwedge^n G)^{-m}\\
    &=\bigoplus_{\lambda \in B(r)} S_{\lambda+(-m^n)+(n^m) + (-n^m)}F  \otimes S_{\lambda+(-m^n)} G \\
    &=\bigoplus_{\lambda \in B(r)} S_{\lambda+(-m^n)}F  \otimes S_{\lambda+(-m^n)} G \\
    &=\bigoplus_{\substack{\lambda \in \ZZ^n_{dom}\\  t-n\geq \lambda_1\\ \lambda_n \geq 0 \\|\lambda|= r }} S_{\lambda}F  \otimes S_{\lambda} G.
\end{align*}
Thus by Cauchy's formula \eqref{eq-Cauchy's formula} we see that $J_t$ as a $\Gl$-representation is a direct sum of terms $S_\mu F \otimes S_\mu$ not present in the above direct sum. Hence by Remark \ref{rmk-Gl-ideals determined by rep structure} and Formula \eqref{eq-ideal generated by single partition} we have that
\[J_t = \bigoplus_{\substack{\lambda \in \ZZ^n_{dom}\\  \lambda_1 \geq t-n+1\\ \lambda_n \geq 0 }} S_{\lambda}F  \otimes S_{\lambda} G =I_{(t-n+1)}.\]
\end{proof}






\subsection*{Comments on Characteristic $p>0$}
The description of these local cohomology modules in characteristic $p>0$ is almost completely unknown. While the results of Section \ref{section-lift} are not dependent on characteristic, the approach used for the $n\times (n-1)$ case fails completely. Since $I$ is Cohen-Macaulay of height $(m-n+1)$, we have that $H^{mn-n^2+1}_I(R)=0$ so extracting information from the maps $\Ext_R^{mn-n^2+1}(R/I^t,R)\to H^{mn-n^2+1}_I(R)$ is challenging.

Computer computations in Macaulay2 \cite{M2} show that in prime characteristic the modules $H_\maxm^{n^2-1}(R/I^t)$ are not always cyclic and may have generators in multiple degrees. In \cite{Kenkel20IsomorphismsBetweenLocalCohomologyModulesAsTruncationsofTaylorSeries} it was shown that the degree $0$ component of $H_\maxm^{n^2-1}(R/I^t)$ can have arbitrarily large vector space dimension, suggesting these modules may have arbitrarily many generators. 

 \bibliographystyle{amsalpha}
  \bibliography{refs}

\end{document}